\documentclass[11pt,parskip=half]{scrartcl}
\usepackage[a4paper,hmargin={2.5cm,2.5cm},vmargin={2.5cm,2.5cm},heightrounded, marginparwidth=2.2cm, marginparsep=0.1cm]{geometry}

\usepackage[utf8]{inputenc}
\usepackage[T1]{fontenc}

\usepackage[leqno]{amsmath}
\usepackage{amssymb,amsthm}
\usepackage{tikz-cd}
\usepackage[all]{xy}

\usepackage{endnotes}  

\theoremstyle{plain}
\newtheorem{theorem}{Theorem}[section]
\newtheorem{lemma}[theorem]{Lemma}
\newtheorem{proposition}[theorem]{Proposition}
\newtheorem{corollary}[theorem]{Corollary}

\theoremstyle{definition}
\newtheorem{definition}[theorem]{Definition}

\theoremstyle{remark}
\newtheorem{remark}[theorem]{Remark}
\newtheorem*{remark*}{Remark}

\RequirePackage[shortlabels]{enumitem}

\newlist{tfae}{enumerate}{1}%
\setlist[tfae,1]{label=(\roman*)}%

\newcommand{\catfont}[1]{\mathsf{#1}}

\newcommand{\SET}{\catfont{Set}}

\DeclareMathAlphabet{\mathmybb}{U}{bbold}{m}{n}

\newcommand{\CA}{\mathcal{A}}
\newcommand{\CB}{\mathcal{B}}
\newcommand{\CC}{\mathcal{C}}
\newcommand{\CD}{\mathcal{D}}
\newcommand{\CE}{\mathcal{E}}

\newcommand{\CS}{\mathcal{S}}

\newcommand{\CV}{\mathcal{V}}

\DeclareMathOperator{\colim}{colim}

\DeclareMathAlphabet{\mathpzc}{OT1}{pzc}{m}{it}

\newcommand{\op}{\mathrm{op}}

\newcommand{\calC}{\mathcal{C}}
\newcommand{\calD}{\mathcal{D}}
\newcommand{\calS}{\mathcal{S}}

\newcommand{\Set}{\SET}

\newcommand{\theory}[1]{\mathcal{#1}}

\newcommand{\inj}{\mathrm{in}}   

\DeclareMathOperator{\Obj}{Obj}

\title{Revisiting Hugo Volger's paper\\
{\em \"{U}ber die Existenz der freien Algebren}}

\author{Mat\'{i}as Menni and Walter Tholen}
\publishers{}
\date{\today}

\usepackage[hypertexnames=false]{hyperref}

\hypersetup{
  colorlinks = true,
  citecolor= [rgb]{0,0.75,0}, 
  urlcolor=[rgb]{0.4,0,0.4}, 
  linkcolor=[rgb]{0.8,0,0} 
}

\begin{document}

\maketitle

\begin{abstract}
We give a modern account of Volger's  1967 paper which, motivated by the construction of free algebras for a Lawvere-Linton theory, gives a very constructive  proof that the left Kan extension of a product-preserving $\SET$-valued functor is product-preserving.  
We also  analyze how it anticipates, and in part even exceeds, subsequent work of the 1970s. 
\end{abstract}

{\em Mathematics Classification:} 18C10, 08B20, 18A40, 18D20.




\section{Introduction}
In the reprint of his celebrated thesis \cite{Lawvere63} (published in part in \cite{Lawvere63a, Lawvere65}), as part of the {\em Author's Comments}, Lawvere writes in reference to Chapter 4 on {\em Algebraic Functors} (see  \cite[Pages 19/20]{Lawvere04}):

\begin{quotation}
 The calculus of algebraic functors and their adjoints is at least as important in practice as the algebraic categories themselves. Thus it is unfortunate that there was no indication of the fundamental fact that these functors preserve reflexive coequalizers (of course their preservation of filtered colimits has always been implicit). As remarked above, “the algebra engendered by a prealgebra” is a special case of an adjoint to a (generalized) algebraic functor. However, contrary to what might be suggested by the treatment in this chapter, the use of that reflection is not a necessary supplement to the use of Kan extensions in proving the general existence of algebraic adjoints: as remarked only later by Michel Andr\'{e}, Jean B\'{e}nabou, Hugo Volger, and others, the special exactness properties of the background category of sets imply that the left Kan extension, along any morphism of algebraic theories, of any algebra in sets, is already itself again product-preserving. Clearly, the same sort of thing holds, for example, with any topos as background. 
\end{quotation}

Especially the mentioning of Volger's name made us curious. Although the paper is not listed in the References of \cite{Lawvere04},
we suppose that Lawvere is alluding to Volger's free-algebra construction in  the paper \cite{Volger68}, which was submitted to {\em Mathematische Zeitschrift} in May of 1967; its German title translates to {\em About the existence of free algebras}. Rather than proving their mere existence (as may easily be done by Adjoint Functor Theorem methods \cite{ML71}), Volger gives a rather concrete construction of free algebras, for theories that allow infinitary operations up to a regular cardinal. Nevertheless, the paper is rarely mentioned anywhere in the subsequent literature pertaining to Lawvere's algebraic theories, including the books or monographs that treat Lawvere theories in a signifcant manner, such as \cite{Pareigis69, Pareigis70} or \cite{PedRov04, ARV11}. Schubert's book \cite{Schubert72} and its German precursor \cite{Schubert70} list Volger's 1967 {\em Diplomarbeit} (Master's thesis) at the University of Freiburg, titled {\em Kategorien von Algebren \"{u}ber algebraischen Theorien} ({\em Categories of algebras over algebraic theories}), on which we presume Volger's paper to be  based (we have no access to his thesis). It is however not apparent to which extent Schubert used Volger's work in his proofs. We found only two articles which directly expand on Volger's paper, namely \cite{Schumacher70} and \cite{HS72}, but like Volger's paper, they hardly get cited in the subsequent literature. (In fact, they did not turn up for us with the search function of {\em Zentralblatt} for papers referencing Volger's paper!) We believe, however, that both, Volger's paper and the Howlett-Schumacher paper \cite{HS72} contain important steps toward Borceux's later work with Day \cite{Borceux76, BorceuxDay77, BorceuxDay80}, which is duly acknowledged in the literature as a pillar of the enriched treatment of algebraic theories and their algebras; see, for example, the often cited paper \cite{KellyLack93}, or the more recent  \cite{KP2012,LucyshynParker24}. In addition, unlike most later works, Volger's construction covers algebras for theories with infinitary operations.

The language barrier and the sometimes cumbersome notation in Volger's paper may have contributed to its limited resonance in the literature and the likely small number of its actual readers. The goal of this article is to make the achievements of Volger's paper easily accessible to a wider audience and to demonstrate their importance.

In Section~\ref{SecWeaklyCofilteredCats} we prove a generalized version of a key technical lemma in Volger's paper. This will help us present his result in a simpler form. 
In Section~\ref{SecVolger} we  present a rather ``liberal'' and shortened translation of Volger's paper in which, while following the original structure of the paper, we use in some instances more modern terminology and simplify some of the original lengthy proofs, thus using methods that have become common place only after the appearance of Volger's paper, especially with the publication of Mac Lane's book \cite{ML71}. In order to preserve the flow of the original text, these deviations from it and our improvements are indicated and explained only in a sequence  of footnotes listed at the end of the paper. 
Then, in Section~\ref{SecComparison}, we briefly compare Volger's result with the presentation of free algebras in Lawvere's thesis, and analyze on one hand to which extent it is, in conjunction with the Howlett-Schumacher paper \cite{HS72}, a true forerunner to Borceux's work with Day, and on the other hand, when applied to $\SET$-based algebras, to which extent it goes beyond these papers since his result includes the consideration of infinitary theories (in the sense of Linton \cite{Linton66}).

The common underlying problem considered in the papers cited above may be formulated as follows:

 {\em Given functors $F:\CB\to \CC$ and $X:\CB\to\CE$ of locally small categories, for $\CB$ small with finite products and $X$ preserving them, and for $\CE$ (small-)cocomplete, which additional conditions are needed to guarantee that then also the left Kan extension $\mathrm{Lan}_FX:\CC\to\CE$ of $X$ along $F$ preserves (any existing) finite products in $\CC$?}
 
The starting point for the work by Borceux and Day  is their elegant short argument that, when $\CE$ is Cartesian closed, no additional condition is needed for an affirmaitive answer to the problem, while Howlett and Schumacher require that $F$ should preserve finite products. Their argumentation adopts parts of Volger's auxiliary proofs which produce an affirmative answer in case $\CE$ is the category $\mathsf{Set}$ and $\CB$ has also pullbacks, in addition to finite products. But in our rendering of Volger's proofs we indicate how the use of the finite-product preservation by $F$ may easily be avoided. Furthermore, we show that the pullback assumption in Volger's paper, which is relevant only when one moves from finite products to products of size below an infinite regular cardinal, may be weakened considerably without losing an affirmative answer of the stated problem in case $\CE=\SET$, but at the expense of some additional restriction on the functor $X$; for details, see Theorem \ref{VolgerTheorem+}.

\section{A generalization of a lemma due to Volger}
\label{SecWeaklyCofilteredCats}

In this section we prove a generalization of the key \cite[Lemma~5]{Volger68}; a result that, today, may be seen as about the calculation of colimits for certain special functors whose categories of elements are weakly cofiltered.

\subsection{Colimits}

We first recall a basic categorical fact. See, for example, \cite[Theorem~{V.2.1}]{ML71}.

 \begin{proposition}[Colimits by coproducts and coequalizers]\label{PropColimits} If the category $\calS$ has small coproducts and coequalizers then, for any small category $\calC$,  every functor ${H : \calC\rightarrow \calS}$ has a colimit and, if we let ${\gamma : H \rightarrow \colim H}$ be the colimiting cocone, then the following diagram  
\[\xymatrix{
\coprod\limits_{f : j \rightarrow k } H j \ar[rr]<+1ex>^-{[\inj_j \mid f ]} \ar[rr]<-1ex>_-{[ \inj_k (H f)\mid  f ]} & & \coprod\limits_{c } H c \ar[rr]^-{[\gamma_c \mid c\in \calC]} && \colim H
}\]
is a coequalizer.
 \end{proposition}

For any category $\calC$, the  relation $\rightsquigarrow_{\calC}$ on the class ${\Obj \calC}$ of objects of $\calC$  is defined by 
\[ x \rightsquigarrow_{\calC} y \]
if, and only, if there is a map ${x \rightarrow y}$ in $\calC$. This relation is clearly reflexive and transitive, but not necessarily symmetric.

For any small category $\calC$ and any functor ${H : \calC \rightarrow \Set}$ let ${1/H}$ be the category `of elements' of $H$.
Its objects are pairs ${(c, x)}$ with $c$ and object in $\calC$ and ${x \in H c}$.
A map ${f : (j, x) \rightarrow (k,y)}$  in $1/H$ is a map ${f : j \rightarrow k}$ in $\calC$ such that ${(H f) x = y}$.

The set ${\Obj(1/H)}$ may be identified with 
 \[ \coprod_{c \in \calC}  H c \]
 so we easily obtain the following.
 
 \begin{corollary}\label{CorForLem2.2}  If $\calC$ is small then the colimit of a functor ${H : \calC \rightarrow \Set}$ may be calculated as the quotient of the set ${\Obj(1/H)}$ by the equivalence relation generated by  $\rightsquigarrow_{1/H}$.
 \end{corollary}
 \begin{proof}
 Immediate from Proposition~\ref{PropColimits}.
 \end{proof}

\subsection{The relation $\Lambda$}

For any category $\calC$, the  relation ${\Lambda_\calC}$  on the class ${\Obj \calC}$ of objects of $\calC$  is defined by 
\[ x \Lambda_{\calC} y\]
if, and only, if  there is a span ${x \leftarrow  c \rightarrow y}$ of maps in $\calC$.
This relation is reflexive and symmetric, but not necessarily transitive.

\begin{lemma}\label{LemTheSameEquivRel} 
For any category $\calC$, the relations ${\rightsquigarrow}$ and $\Lambda_\calC$ generate the same equivalence relation.
\end{lemma}
\begin{proof}
To lighten the notation we do not write the category as a subindex.
On the one hand, if the map ${f : x \rightarrow y}$ in $\calC$ witnesses that  ${x \rightsquigarrow  y }$, then the span 
 \[\xymatrix{ x & \ar[l]_-{1_x} x \ar[r]^-f & y }\]
 witnesses that ${x \Lambda y}$.  On the other hand, if ${x \Lambda y}$ is witnessed by a span ${x \leftarrow c \rightarrow y}$ in $\calC$ then  ${c \rightsquigarrow x}$ and ${c \rightsquigarrow y}$.
\end{proof}

 At the risk of being obvious, we remark that for a functor  ${H : \calC \rightarrow \Set}$,
 \[  (c, x) \Lambda_{1/H} (d, y) \]
 if, and only if, there is a span 
 \[ \xymatrix{(c, x)  & \ar[l]_-{f} (e, z) \ar[r]^-g & (d, y) } \]
 in ${1/H}$.  That is, if there is a span ${\xymatrix{c & \ar[l]_-{f} e \ar[r]^-g & d }}$ and an element ${z \in H e}$ such that ${(H f) z = x}$ and ${(H g) z = y}$.

\begin{lemma}\label{Lem2.2} 
 If $\calC$ is small then the colimit of a functor ${H : \calC \rightarrow \Set}$ may be calculated as the quotient of the set ${\Obj(1/H)}$ by the equivalence relation generated by  $\Lambda_{1/H}$.
\end{lemma}
 \begin{proof}
 Follows from Corollary~\ref{CorForLem2.2} and  Lemma~\ref{LemTheSameEquivRel}.
 \end{proof}

\subsection{Weakly cofiltered categories}

In this section we prove a generalization of a technical result in \cite{Volger68} which apperas in our freely edited translation as Lemma~\ref{Lemma5}.

\begin{definition}\label{DefWeaklyCofiltered}
A category is {\em weakly cofiltered} if every cospan may be completed to a commutative square.
\end{definition}

\begin{remark}\label{RemOnTransitivity} If $\calC$ is weakly cofiltered then $\Lambda_\calC$ is transitive and hence, an equivalence relation.
\end{remark}

In order to stay in harmony with Volger’s notation \cite{Volger68}, for a functor  ${F : \theory{B} \rightarrow \calC}$ we let
\[ \xymatrix{ F/c \ar[rr]^-{V_{F,c}} & & \theory{B} }\]
denote the `forgetful' functor.

\begin{proposition}\label{PropLem3.2} 
Let ${G :  \theory{B} \rightarrow \calD}$ be any functor with $\calD$ locally small.
For any object  $d$ in $\calD$, the following are equivalent:
\begin{enumerate}
\item The category ${d/G}$ is weakly cofiltered.
\item For every  ${F : \theory{B} \rightarrow \calC}$ to a locally small category, and every object $c$ in $\calC$, the category of elements of the composite
\[ \xymatrix{
F/c \ar[r]^-{V_{F,c}} & \theory{B} \ar[r]^-G & \calD \ar[rr]^-{\calD(d,-)} && \Set
}\]
is weakly cofiltered.
\end{enumerate}
\end{proposition}
\begin{proof}
The second item implies the first by taking $F$ to be the unique functor from $\theory{B}$ to the terminal category.
To prove the converse let  ${P_{c, d} : F/c \rightarrow \Set}$ be the composite functor  in the statement.
The category  ${(F/c)/P_{c,d}}$ may be described as follows. Its objects are triples ${(b, u, v)}$ with ${u : F b \rightarrow c}$ in $\calC$, and ${v : d \rightarrow G b}$ in $\calD$. 
A map ${h : (b_0, u_0, v_0) \rightarrow (b_1, u_1, v_1)}$  is a map ${h : b_0 \rightarrow b_1}$ in $\theory{B}$ such that ${u_1 (F h) = u_0}$    and  ${(G h) v_0 = v_1}$.
Equivalently, such that  ${h : (b_0, u_0) \rightarrow (b_1, u_1)}$ in $F/c$ and  ${h : (b_0, v_0) \rightarrow (b_1, v_1)}$ in ${d/G}$.

To prove that ${(F/c)/P_{c,d}}$ is weakly cofiltered let  ${h_0 : (b_0, u_0, v_0) \rightarrow (b, u, v)}$ and ${h_1 : (b_1, u_1, v_1) \rightarrow (b, u, v)}$.
Since ${d/G}$ is weakly cofiltered by hypothesis, there is a commutative square
\[\xymatrix{
(b', v') \ar[d]_-{k_0} \ar[r]^-{k_1} & (b_1, v_1) \ar[d]^-{h_1} \\
(b_0, v_0) \ar[r]_-{h_0} & (b, v)
}\]
in ${d/G}$. Now observe that the following diagram commutes
\[\xymatrix{
 & \ar@(l,u)[ld]_-{F k_0} F b' \ar@(r,u)[rd]^-{F k_1} \\
F b_0 \ar@(d,l)[rd]_-{u_0} \ar[r]^-{F h_0} & F b \ar[d]^-{u} & \ar[l]_-{F h_1}  \ar@(d,r)[ld]^-{u_1}  F b_1 \\
 & c &
}\]
and let ${u' : F b' \rightarrow c}$ be the resulting map.
We then have the object ${(b', u', v')}$ in ${(F/c)/P_{c,d}}$ and, from the same diagram we infer that
 ${k_0 : (b', u') \rightarrow (b_0, u_0)}$ and ${k_1 : (b', u') \rightarrow (b_1, u_1)}$  in $F/c$.
\end{proof}

We next discuss a simple sufficient condition for the first item of Proposition~\ref{PropLem3.2} to hold.
 
 \begin{definition}\label{DefWPC} 
 A functor  $G:\CB\to\CD$  satisfies the {\em Weak Pullback Condition (WPC)} if every cospan $\xymatrix{\cdot\ar[r]^f &\cdot &\cdot\ar[l]_g}$  in   $\CB$ may be completed to a commutative diagram 
$$\xymatrix{& \cdot\ar[ld]_p\ar[rd]^q &\\
\cdot\ar[rd]_f && \cdot\ar[ld]^g\\
& \cdot \\
}$$
that $G$ transforms into a weak pullback diagram  in $\CD$.
 \end{definition}

We note that $G$ trivially satisfies WPC when $\CB$ has pullbacks and $G$ preserves them weakly, {\em i.e.}, transforms them into weak pullback diagrams.

\begin{lemma}\label{LemWPC} If  $G:\CB\to\CD$ satisfies the WPC then ${d/\CD}$ is weakly cofiltered for all $d\in\CD$.
\end{lemma}
\begin{proof}
Easy.
\end{proof}

 For a span of functors
$$\xymatrix{\CC & \CB\ar[l]_-{F}\ar[r]^-{G} & \CD}$$
 and  specified objects $c$ and $d$ in the locally small categories $\CC$ and $\CD$, the composite
 \[ \xymatrix{F/c\ar[rr]^-{V_{F,c}} && \CB\ar[rr]^-G && \CD\ar[rr]^-{\CD(d,-) } && \SET} \]
 will be denoted by 
 \[  P_{c,d} : F/c \rightarrow  \SET \]
 as in the proof of Proposition~\ref{PropLem3.2}.

\begin{lemma}\label{Lemma5WPCnew}
 Let $\CB$ be a small category, and consider a span of functors
\[\xymatrix{c\in \CC & \CB\ar[l]_-{F}\ar[r]^-{G} & \CD\ni d} \]
with specified objects $c$ and $d$ in the locally small categories $\CC$ and $\CD$, respectively.
If $G$ satisfies the WPC then the relation ${\Lambda_{c,d} = \Lambda_{1/P_{c, d}}}$ is an equivalence relation.
Hence, 
\[ \mathrm{colim}\,P_{c,d}\cong \left(\coprod_{b\in \CB}\CD(d,Gb)\times \CC(Fb,c)\right)/ \Lambda_{c,d} \,.\]
\end{lemma}
\begin{proof}
Lemma~\ref{LemWPC} implies that the category $d/\CD$ is weakly cofiltered for every ${d \in \CD}$.
So the category $1/P_{c,d}$  of elements of $P_{c,d}$ is weakly cofiltered by Proposition~\ref{PropLem3.2}.
Then $\Lambda_{c,d}$ is an equivalence relation by Remark~\ref{RemOnTransitivity}, and the colimit formula holds by Lemma~\ref{Lem2.2}.
\end{proof}

\section{Our reading of Volger's article -- a freely edited translation}
\label{SecVolger}

In this article, like for Lawvere  \cite{Lawvere63} and Linton \cite{Linton66},  {\em algebras} are functors with values in $\SET$, the category of all sets, defined on certain small categories, called algebraic theories. Differently from Lawvere \cite{Lawvere63}, the classes of algebras may have uncountable presentation rank, but differently from Linton \cite{Linton66} a presentation rank has to exist.

Let $r$ be an (infinite) regular cardinal and denote by $\CS_r$ a skeleton of the category of sets of cardinality $<r$, considered as a full subcategory of the category $\SET$ of all sets and functions. An \textit{algebraic theory (of presentation rank\endnote{The term ``presentation rank'' used here is called ``dimension'' by Volger.} $r$)} is a small category $\CA$ which comes with a functor $A:\CS_r^\op\to\CA$ that is bijective on objects and preserves products\endnote{Volger works with a bijective-on-objects and coproduct-preserving functor $A:\CS_r\to\CA$, so his notation of an algebraic theory is formally dual to the one used more frequently, as adopted also here.}. We write $A^k$ for $A(k)$ for all $k<r$ and denote\endnote{Volger uses Lawvere's notation $\SET^{(\CA)}$ for $\mathsf{Alg}(\CA)$.}
by
$$\mathsf{Alg}(\CA)$$
the full subcategory of the functor category $\SET^\CA$ containing all product-preserving functors $X:\CA\to\SET$. In generalization of an idea of Andr\'{e} for the construction of free algebras in case $r=\aleph_0$, the goal is to establish the left adjoint to the (forgetful) evaluation functor 
$$U:\mathsf{Alg}(\CA)\longrightarrow\SET,\quad X\longmapsto X(A^1)\,.$$

Note\endnote{This paragraph includes snippets appearing only later in Volger's article, in the proof of the main theorem.} that, for the trivial algebraic theory $\CS_r^\op$ (equipped with its identity functor), the forgetful functor $\mathsf{Alg}(\CS_r^\op)\to \SET$ is an equivalence of categories. Its (left) adjoint interprets a set $X$ as the $\mathsf{Alg}(\CS_r^\op)$-object given by its restricted hom 
$$X^{(-)}=\SET(I^{\op}(-),X): \CS_r^\op\longrightarrow\SET,\quad m\longmapsto X^m,$$
 where $I$ is the inclusion functor $\CS_r\hookrightarrow\SET$. The strategy is to describe the desired adjunction
$$\xymatrix{\mathsf{Alg}(\CA)\ar@/^0.6pc/[rr]^{U\quad\;\;} &\top& \SET\simeq \mathsf{Alg}(\CS_r^\op)\ar@/^.4pc/[ll]
}$$
as a restriction of the adjunction
$$\xymatrix{\SET^{\CA}\ar@/^.6pc/[rr]^{(-)A} &\top& \SET^{\CS_r^\op}\;.\ar@/^.4pc/[ll]^{\mathrm{Lan}_A}
}$$
The task is then to show that the left Kan extension of $X^{{\scriptscriptstyle (-)}}:\CS_r^\op\to\SET$ along $A:\CS_r^\op\to\CA$ preserves products\endnote{We use here Mac Lane's terminology and notation for left Kan extensions, established later than Volger's paper.}. This is achieved with the help of five lemmata.

First some notation: for a functor $F:\CB\to\CC$ and an object $c\in\CC$, let $F/c$ denote the comma category of objects $(b,z)$ with $z:Fb\to c$ in $\CC$ and morphisms $f:(b,z)\to(b',z')$ with $f:b\to b'$ in $\CB$ and $z'\cdot Ff=z$ in $\CC$. 


\begin{lemma}\label{Lemma1}{{\em (Lawvere \cite{Lawvere63}, I.2.5)\endnote{Other than Lawvere's thesis, for this  (at the time not yet widely known) lemma, Volger cites also Ulmer \cite{Ulmer66}, presumably a precursor to \cite{Ulmer68}. 
}}}
For a functor $F:\CB\to\CC$ with $\CB$  small, $\CC$ locally small, and $\CE$ small-cocomplete, the functor $(-)F:\CE^\CC\to\CE^\CB$ has a left adjoint, $\mathrm{Lan}_F$, which assigns to $X\in\CE^\CB$ the functor 
$$\mathrm{Lan}_FX:\CC\longrightarrow\CE,\quad c\longmapsto \mathrm{colim}\,XV_{F,c}\;.$$
\end{lemma}

Also the second\endnote{Volger gives the explicit colimit formula, not referring to the ``element category''.} and third\endnote{We nevertheless include here a brief sketch of the proof.} Lemma are stated without proof.
\begin{lemma}\label{Lemma2}
For $\CB$ small, the colimit of a functor $P:\CB\to\SET$ may be given by the connected components of the element category $1/P$ of $P$; explicitly,
$$\mathrm{colim}\,P=\left(\coprod_{b\in\CB}Pb\right)/\!\sim\;,$$
where $\sim$ is the transitive hull of the relation $R$ on $\coprod_{b\in\CB}Pb$ defined by
$$(b,x) \,R\,(b',x')\iff \exists\;(\xymatrix{(b,x) & (\overline{b},\overline{x})\ar[l]_{\;f}\ar[r]^{g\;}& (b',x')}) \text{ in } 1/P\,.$$
\end{lemma}
\begin{proof}
(Notice that this is Lemma~\ref{Lem2.2}.)
\end{proof}


\begin{lemma}\label{Lemma3}
For $\CB_i$ small, assume that the relation $R_i$ defined by the functor $P_i:\CB_i\to \SET$ as in Lemma {\em \ref{Lemma2}} is transitive, $i=1,2$\endnote{Later in his paper, Volger mentions, and uses, the fact that the indexing system $\{1,2\}$ may be replaced by any set (including the empty set: see the first footnote to Lemma \ref{Lemma4}).}. Then,
 for the functor\endnote{We keep here Volger's notation, although this is not to be read as a ``tensor product''.}
 $$P_1\otimes P_2=( \xymatrix{\CB_1\times\CB_2\ar[rr]^{P_1\times P_2} &&\SET\times\SET\ar[r]^{\quad\times}&\SET}),\quad (b_1,b_2)\mapsto P_1b_1\times P_2b_2, $$
  the canonical map 
$$k:\mathrm{colim}( P_1\otimes P_2)\longrightarrow \mathrm{colim}\,P_1\times\mathrm{colim}\,P_2$$
is bijective.
\end{lemma}

\begin{proof}
Note that the assumed transitivity makes $R_i$ an equivalence relation on $\coprod_{b\in\CB_i}P_ib$ for $i=1,2$. The relation $R$ induced by $P_1\otimes P_2$ on $\coprod_{(b_1,b_2)}P_1b_1\times P_2b_2$ clearly satisfies
$$((b_1,b_2),(x_1,x_2))\;R\;((b_1',b_2'),(x_1',x_2'))\iff (b_1,x_1)\,R_1\,(b_1',x_1') \text{ and } (b_2,x_2)\,R_2\,(b_2',x_2').$$
Therefore, $R$ is an equivalence relation as well, and the canonical map $k$ is bijective.
\end{proof}

The following lemma\endnote{Only the colimit formula in the second part of the formulation given here appears in Volger's article, which he proves without explicit recourse to the functor $W$, using however the same ingredients arising from the products in $\CB$ as used in the proof given here. With the proof of the implicitly stated colimit invariance of $W$, Volger's proof of the Lemma takes $3\frac{1}{2}$ pages of his $8\frac{1}{2}$-page paper.} originates with an idea of Andr\'{e}:

\begin{lemma}\label{Lemma4}
Let the small category $\CB$ have binary\endnote{Volger remarks later that, as in Lemma \ref{Lemma3}, one may work with $I$-fold products for any set $I$, rather than just with binary ones. He mentions that the case $I=\emptyset$ requires separate considerations. But when read in the form given here, his proof works also in that case: replacing $c_1\times c_2$ by the empty product $1$ in $\CC$, one has $F/1\cong \CB$ and then, up to isomorphism,  $W$ becomes the functor $\mathbf 1\to\CB$ that picks the terminal object in $\CB$. As a right adjoint, W is final: see Lemma~\ref{HSLemma4}.} products, and let the functor $F:\CB\to\CC$ preserve\endnote{The preservation hypothesis on $F$ may be avoided; see Lemma \ref{HSLemma4}.} them.  Then for all objects $c_1,c_2$ with a product in the locally small category $\CC$ there is a final\endnote{``Final'' is often, but misleadingly, called ``cofinal", apparently as a questionable contraction of the German term {\em konfinal},  as used by Gabriel and Ulmer in Definition 2.12. of their \cite{GU71}.} functor $W$ making the diagram
$$\xymatrix{F/c_1\times F/c_2\ar[rr]^W\ar[rd]_{V_{F,c_1}\otimes V_{F,c_2}} && F/(c_1\times c_2)\ar[ld]^{V_{F,c_1\times c_2}}\\
& \CB &\\
}$$
commute. Therefore, for any functor $X:\CB\to\CE$ one has
$$ \mathrm{colim}(XV_{F,c_1\times c_2})\cong\mathrm{colim}(X(V_{F,c_1}\otimes V_{F,c_2}))\,,$$
existence of these colimits in $\CE$ granted.
\end{lemma}

\begin{proof}
Let $W$ map the object $((b_1, z_1),(b_2, z_2))$  in $F/c_1\times F/c_2$ to $(b_1\times b_2, (z_1\times z_2)\cdot i_{b_1,b_2})$ in $F/(c_1\times c_2)$, with the canonical isomorphism $i_{b_1,b_2}:F(b_1\times b_2)\to Fb_1\times Fb_2$, and we extend this assignment canonically to morphisms. To confirm finality of $W$ we must show that the category $(b,z)/W$ is connected, for every object $(b,z)$ in $F/(c_1\times c_2)$. In fact, we show that $(b,z)/W$ has a weakly initial object, as follows.

For the given object $(b,z)$ with $z: Fb\to c_1\times c_2$ in $\CC$, consider the left-hand side of the following diagram, with product projections $\pi_i$ of $c_1\times c_2$ (i=1,2), and with diagonal morphisms $\Delta_b$ and $\Delta_{Fb}$; it constitutes a morphism $\Delta_b: (b,z)\to W((b,\pi_1\cdot z),(b,\pi_2\cdot z))$ in $F/(c_1\times c_2)$, that is, an object in $(b,z)/W$.
$$\xymatrix{ && Fb\ar[lld]_{F\Delta_b}\ar[rrd]^{Fu}\ar|(.33)\hole|(.66)\hole[ddd]_z\ar|(.5)\hole[lldd]|(.65){\Delta_{Fb}} &&\\
F(b\times b)\ar[d]_{i_{b,b}}\ar[rrrr]^(.67){F((p_1\cdot u)\times (p_2\cdot u))} &&&& F(b_1\times b_2)\ar[d]^{i_{b_1,b_2}}\\
Fb\times Fb\ar[rrrr]^(.68){F(p_1\cdot u)\times F(p_2\cdot u)}\ar[rrd]_{(\pi_1\cdot z)\times(\pi_2\cdot z)\quad} &&&& Fb_1\times Fb_2\ar[lld]^{z_1\times z_2}\\
&& c_1\times c_2 &&\\
}$$
To show that this object $(((b,\pi_1\cdot z),(b,\pi_2\cdot z)),\Delta_b)$ is (weakly) initial, consider any other object $(((b_1,z_1),(b_2,z_2)),u)$ in $(b,z)/W$, as displayed by the right-hand side of the diagram above. As the entire diagram commutes, with the product projections $p_i$ of $b_1\times b_2$ we have the morphism $(p_1\cdot u,p_2\cdot u) $ in $F/c_1\times F/c_2$ making the diagram
$$\xymatrix{&&(b,z)\ar[lld]_{\Delta_b}\ar[rrd]^u && \\
W((b,\pi_1\cdot z),(b,\pi_2\cdot z))\ar[rrrr]^{\quad W(p_1\cdot u,p_2\cdot u)} &&&& W((b_1,z_1),(b_2,z_2))\\
}$$
commute in  $F/(c_1\times c_2)$. This gives the desired morphism in $(b,z)/W$, namely
$$(p_1\cdot u,p_2\cdot u):(((b,\pi_1\cdot z),(b,\pi_2\cdot z)),\Delta_b)\to(((b_1,z_1),(b_2,z_2)),u)\,.$$ 
\end{proof}

Here is the fifth lemma\endnote{We regard this lemma as the most original contribution of Volger's paper. Under its pullback hypotheses, the Lemma actually gives a simplified formula for the evaluation of the profunctor (or distributor) ${G^*\circ F_*:\xymatrix{\CC\ar[r]|\circ & \CD}}$ at $(d,c)$; see \cite{Borceux94}.} that is used for the proof of the main theorem:

\begin{lemma}\label{Lemma5} Let $\CB$ be a small category with pullbacks, and consider a span of functors\endnote{Guided by the intended application of the Lemma, Volger assumes (unnecessarily) that $G$ be an inclusion functor but neglects to mention its needed pullback preservation and, in the proof, he actually assumes $G$ to be an identity functor. The assumptions on pullbacks were corrected in the {\em erratum} \cite{Volger69} (as the only needed correction), and they may actually be weakened considerably: see Lemma \ref{Lemma5WPCnew}.}
$$\xymatrix{c\in \CC & \CB\ar[l]_-{F}\ar[r]^-{G} & \CD\ni d}$$
with $G$ preserving pullbacks, and with specified objects $c$ and $d$ in the locally small categories $\CC$ and $\CD$, respectively. Then, for  the functor
$$P_{c,d}=(\xymatrix{F/c\ar[rr]^{V_{F,c}} && \CB\ar[rr]^G && \CD\ar[rr]^{\CD(d,-)} && \SET \\
})$$
the relation ${\Lambda_{c,d} = \Lambda_{1/P_{c, d}}}$ is an equivalence relation, and\endnote{Volger states the resulting colimit formula only in the main theorem, in the special situation considered there.} so
$$\mathrm{colim}\,P_{c,d}\cong \left(\coprod_{b\in \CB}\CD(d,Gb)\times \CC(Fb,c)\right)/\Lambda_{c,d}\,.$$
\end{lemma}

\begin{proof} 
(We do not reproduce Volger's proof since it now follows from the more general Lemma~\ref{Lemma5WPCnew}.)
\end{proof}

\begin{theorem}\label{VolgerTheorem}
For an algebraic theory $A:\CS^{\op}_r\to\CA$ of presentation rank $r$, the free algebra over a set $X$ is the functor assigning to an object $A^m\in\CA\;(m<r)$ the set
$$\left(\coprod_{n<r}\CA(A^n,A^m)\times X^n\right)/R_m\;,$$
to be naturally extended to the morphisms of $\CA$; here $R_m$ is the equivalence relation on the set of all triples $(n,\omega, x)$ with $n<r$, a (formal $n$-ary and $m$-valued) operation 
$\omega: A^n\to A^m$ in $\CA$, and with $x\in X^n$ (also written as a map $x:n\to X$), defined\endnote{The explicit description has been added here to Volger's formulation of his theorem.} by

\begin{tabular}{ll}
$(n,\omega,x)\,R_m\,(n',\omega',x')\iff\exists$ & $\xymatrix{n\ar[r]^{\varphi} & \overline{n} & n'\ar[l]_{\psi}
}\mbox{ in }\CS_r \mbox{ and }\overline{x}\in X^{\overline{n}}$ such that the diagrams\\
& $\xymatrix{n\ar[r]^{\varphi}\ar[rd]_x &\overline{n}\ar[d]^{\overline{x}} & n'\ar[l]_{\psi}\ar[ld]^{x'} & A^n\ar[rd]_{\omega} & A^{\overline{n}}\ar[l]_{A^{\varphi}}\ar[r]^{A^{\psi}} & A^{n'}\ar[ld]^{\omega'} \\ 
& X & & & A^m & \\
}$ \quad commute.  \\
\end{tabular}
\end{theorem}

\begin{proof}
Specializing $$\xymatrix{c\in \CC & \CB\ar[l]_{\quad F}\ar[r]^{G\quad} & \CD\ni d}$$ of Lemma \ref{Lemma5} to
$$\xymatrix{A^m\in \CA & \CS_r^{\op}\ar[l]_{\quad A}\ar[r]^{I^{\op}\quad} & \SET^{\op}\ni X}$$
with the inclusion functor $I:\CS_r\to\SET$,
let us first note that, since $\CS_r^{\op}$ has pullbacks preserved by $I^{\op}$, the relation $R_m$ induced by $$P_{A^m,X}:=\SET(-,X)I^{\op}V_{A,A^m}=X^{{\scriptscriptstyle (-)}}V_{A,A^m}$$ on the set
$\coprod_{n<r}\CA(A^n,A^m)\times X^n$ via Lemma \ref{Lemma2} is, by Lemma \ref{Lemma5}, in fact an equivalence relation. 
This gives the desired formula
$$( \mathrm{Lan}_AX^{{\scriptscriptstyle (-)}})(A^m)\cong \left( \coprod_{n<r}\CA(A^n,A^m)\times X^n\right)/R_m\;,$$
as well as the description of $R_m$ as stated in the Theorem.

For some $k<r$ one now considers a family of objects $A^i\in \CA\; (i<k)$ and, according to the introductory remarks, one must show that the functor $\mathrm{Lan}_AX^{{\scriptscriptstyle (-)}}:\CA\to\SET$ preserves their product $\prod_{i<k}A^i$ in $\CA$. Of course, each relation $R_i$ induced by $X^{{\scriptscriptstyle (-)}}V_{A,A^i}$ on the set
$\coprod_{n<r}\CA(A^n,A^i)\times X^n$ is, like $R_m$, an equivalence relation. This allows one to apply Lemma \ref{Lemma3} for every $i<k$ in the  penultimate step of the following calculation, which validates the claimed product preservation:
\begin{align*}
(\mathrm{Lan}_AX^{{\scriptscriptstyle (-)}})\left(\prod_{i<k}A^i \right) &\cong \mathrm{colim}(X^{{\scriptscriptstyle (-)}}V_{A,\prod_{i<k}A^i}) & \text{(Lemma \ref{Lemma1})}\\
& \cong \mathrm{colim}\left(X^{{\scriptscriptstyle (-)}}\bigotimes_{i<k}V_{A,A^i}\right) &\text{(Lemma \ref{Lemma4}, Footnote 12)}\\
& \cong \mathrm{colim}\left(\bigotimes_{i<k}X^{{\scriptscriptstyle (-)}}V_{A,A^i}\right) &(X^{{\scriptscriptstyle (-)}}\text{ preserves products})\\
&\cong\prod_{i<k}\mathrm{colim}(X^{{\scriptscriptstyle (-)}}V_{A,A^i})& \text{(Lemma \ref{Lemma3}, Footnote 9)}\\
&\cong\prod_{i<k} (\mathrm{Lan}_AX^{{\scriptscriptstyle (-)}})(A^i)\;.& \text{(Lemma \ref{Lemma1})}
 \end{align*}
\end{proof}
Literature: \cite{Lawvere63, Lawvere63a, Lawvere65, Linton66, Mitchell65, Ulmer66}.

\section{Some observations and a brief look at some subsequent work}
\label{SecComparison}

Since free algebras determine a left adjoint, it necessarily has a presentation as an instance of  the Adjoint Functor Theorems as done, for example,  in \cite[Section {V.6}]{ML71} but, naturally, the pioneers looked for more informative constructions.

\subsection{Comparison with Lawvere's thesis}
As Lawvere showed in his thesis (\cite{Lawvere63}, see p. 74 of \cite{Lawvere04}), the hom-functor $\CA(A^1,-):\CA\to\SET$ serves as the free $\CA$-algebra $\mathrm{Fr}(1)$ on one generator (with respect to an algebraic theory $\CA$ of presentation rank $r=\aleph_0$, but likewise also when $r>\aleph_0$). Indeed, in effect establishing the Yoneda Lemma, he noticed that, for any $Y\in\mathsf{Alg}(\CA)$ and the forgetful $U:\mathsf{Alg}(\CA)\to\SET$, one has the natural isomorphisms
$$\mathsf{Alg}(\CA)(\mathrm{Fr}(1),Y)=\SET^{\CA}(\CA(A^1,-),Y)\cong Y(A^1)= UY\cong\SET(1,UY)\;,$$
which then served Linton \cite{Linton66} as the starting point for his study of infinitary algebras. Every $n$-ary formal operation $\omega:A^n\to A^1$ in $\CA$ induces the $n$-ary operation
$$(U\mathrm{Fr}(1))^n=\CA(A^1,A^1)^n\cong\xymatrix{\CA(A^1,A^n)\ar[rr]^{\CA(A^1,\,\omega)} && \CA(A^1,A^1)}= U\mathrm{Fr}(1)\;,\;\theta\longmapsto \omega\cdot\theta.$$
on the set $U\mathrm{Fr}(1)$. 
As Lawvere mentions further, the free $\CA$-algebra $\mathrm{Fr}(X)$ of $X$-many generators for any set $X$ is then given as the copower $X\bullet\mathrm{Fr}(1)$ formed in $\mathsf{Alg}(\CA)$.

How does Volger's presentation of free $\CA$-algebras as in Theorem \ref{VolgerTheorem} differ from the Lawvere-Linton presentation? For $X=1$, Theorem \ref{VolgerTheorem} presents
$U\mathrm{Fr}(1)$ as the set 
$(\bigcup_{n<r}\CA(A^n, A^1))/\!\sim$ where 
$$(n,\omega)\sim(n',\omega')\iff \exists \xymatrix{(n\ar[r]^{\varphi} & \overline{n} & n')\ar[l]_{\psi}
}\text{in }\CS_r\text{ such that }\xymatrix{
A^n\ar[rd]_\omega & A^{\overline n}\ar[l]_{A^{\varphi}}\ar[r]^{A^\psi}& A^{n'}\ar[ld]^{\omega'} \\
&A^1&
}\text{commutes.}$$
Indeed, the natural insertion of $\CA(A^1,A^1)$ into $(\bigcup_{n<r}\CA(A^n, A^1))/\!\sim$ has an inverse which assigns to the $\sim$-equivalence class $[n,\omega]$ of $(n,\omega)$ the formal unary operation $\xymatrix{A^1\ar[r]^{\Delta_n} & A^n\ar[r]^{\omega} & A^1}$.

For arbitrary sets $X$, rather than using copowers in $\mathsf{Alg}(\CA)$ or a left adjoint of $\mathsf{Alg}(\CA)\hookrightarrow\SET^{\CA}$, Volger presents $U\mathrm{Fr}(X)$ in direct generalization of his presentation of $U\mathrm{Fr}(1)$, as laid out in Theorem \ref{VolgerTheorem}. Using $\sim$ again to denote the relation $R_1$ of Theorem \ref{VolgerTheorem} and brackets for its equivalence classes, one obtains that every formal operation $\omega:A^m\to A^1$ induces the $m$-ary operation on the set  $(\bigcup_{n<r}\CA(A^n,A^1)\times X^n)/\!\sim$, defined by
$$([n_i,\omega_i,x_i])_{i<m}\quad\longmapsto\quad \left[\; \sum_{i<m}n_i,\;\omega\cdot\prod_{i<m}\omega_i,\;\mu_X((x_i)_{i<m})\; \right] $$
where $\mu_X$ concatenates the family $(x_i)_{i<m}\in\prod_{i<m}X^{n_i}$, {\em i.e.}, in the finitary case  $\mu_X$ is the multiplication of the list monad on $\SET$ at the set $X$.

We also note that, with reference to Linton \cite{Linton66} and Volger \cite{Volger68}, in his rarely cited paper \cite{Schumacher70} Schumacher showed shortly afterwards that the free $\CA$-algebras generated by an arbitrary set are $r$-directed colimits of free $\CA$-algebras with less than $r$-many generators. His paper may therefore be seen as one of the precursors to the important monograph \cite{GU71} by Gabriel and Ulmer on locally presentable categories.

\subsection{Comparison with the Howlett-Schumacher paper}
Howlett and Schumacher in \cite{HS72} presented an important modification of Volger's result, a strengthening of  which we record below as Theorem \ref{HSTheorem}. First we note that, with their arguments, one obtains the following stronger version of Lemma \ref{Lemma4}. 

\begin{lemma}\label{HSLemma4}
Let the small category $\CB$ have products of families of size $<r$, and let
$F:\CB\to\CC$ be a functor\endnote{Unlike in \cite{HS72}, no preservation of products by $F$ is assumed.}, $\CC$ locally small.
 Then for all families of objects $c_i$ in $\CC\;(\;i\leq m<r)$ having a product in $\CC$, there is a final functor $W$ making the diagram
$$\xymatrix{\prod_iF/c_i\ar[r]^W\ar[d]_{\prod_iV_{F,c_i}} & F/\prod_ic_i\ar[d]^{V_{F,\prod_ic_i}}\\
\CB^m\ar[r]^{\prod_{i\leq m}}  & \CB\\
}$$
commute (the diagonal of which is $\bigotimes_iV_{F,c_i}$ in Volger's notation). Therefore, for any functor $X:\CB\to\CE$ one has
$ \mathrm{colim}(XV_{F,\prod_ic_i})\cong\mathrm{colim}\,X(\bigotimes_iV_{F,c_i}),$
existence of these colimits in $\CE$ granted.
\end{lemma}

\begin{proof}
The object part of the functor $W$ of Lemma \ref{Lemma4}, generalized from binary to products of less than $r$ factors, may be described without the use of any products in $\CC$ other than the given $\prod_ic_i$, if we let $(b_i,z_i)_i\in\prod_iF/c_i$ be mapped to $(\prod_ib_i,z)_i\in F/(\prod_ic_i)$, where $z:F(\prod_ib)\to \prod_ic_i$ in $\CC$ is induced by the morphisms $z_i\cdot Fp_i$, with product projections $p_i$ in $\CB$. It easy to see that that $W$ has a (right-inverse) left adjoint $L$, described by 
$$L: F/\prod_ic_i\longrightarrow\prod_iF/c_i,\quad(b,z)\longmapsto (\pi_i\cdot z)_i,$$
with product projections $\pi_i$ in $\CB$. As a right adjoint, $W$ is final\endnote{Howlett and Schumacher \cite{HS72} credit \cite{Berthiaume69} for the implication (right adjoint $\Longrightarrow$ final).}.
\end{proof}

In conjunction with Lemma \ref{Lemma3} we can now generalize Theorem \ref{VolgerTheorem} in the following form:

\begin{theorem}\label{VolgerTheorem+}
Let $F:\CB\to\CC$ and $G:\CB\to \CD$ be functors, with $\CB$ small and $\CC,\CD$ locally small. If $G$ satisfies {\em (WPC)}, and if $\CB$ has products of size $<r$ that are preserved by 
$G$, then also ${\mathrm{Lan}}_FX_d:\CC\to\SET$ preserves all such products in $\CC$, where $X_d=\CD(d,G(-))$ with any $d\in\CD$.
\end{theorem}

\begin{proof}
In the setting of Lemma \ref{HSLemma4} one just adjusts the calculation in the proof of Theorem \ref{VolgerTheorem}:
\begin{align*}
(\mathrm{Lan}_FX_d)\left(\prod_{i<k}c_i\right) &\cong \mathrm{colim}(X_dV_{F,\prod_{i<k}c_i}) & \text{(Lemma \ref{Lemma1})}\\
& \cong \mathrm{colim}\left(X_d\bigotimes_{i<k}V_{F,c_i}\right) &\text{(Lemma \ref{HSLemma4})}\\
& \cong \mathrm{colim}\left(\bigotimes_{i<k}X_dV_{F,c_i}\right) &(X_d\text{ preserves products})\\
&\cong\prod_{i<k}\mathrm{colim}(X_dV_{F,c_i})& \text{(Lemma \ref{Lemma5WPCnew} and  \ref{Lemma3})}\\
&\cong\prod_{i<k} (\mathrm{Lan}_AX_d)(c_i)\;.& \text{(Lemma \ref{Lemma1})}
 \end{align*}
\end{proof}






Howlett and Schumacher \cite{HS72} generalize Volger's Theorem \ref{VolgerTheorem} insofar as algebras no longer need to be $\SET$-based, at the expense of the constraint that the theory must be finitary, {\em i.e.} $r\leq\aleph_0$. With this restriction they replace Lemma \ref{Lemma3} by: 

\begin{lemma}\label{HSLemma3}
Let $\CE$ be a category with finite products and small colimits, such that ${C\times(-):\CE\to\CE}$ preserves colimits for every $C\in\CE$. Then, 
for any functors $P_i:\CB_i\to \CE$ of small categories $\CB_i,\,i=1,2 $, and with
$P_1\otimes P_2=( \xymatrix{\CB_1\times\CB_2\ar[rr]^{P_1\times P_2} &&\CE\times\CE\ar[r]^{\quad\times}&\CE})$
as in Lemma {\em\ref{Lemma3}},  the canonical map 
$$\mathrm{colim}( P_1\otimes P_2)\longrightarrow \mathrm{colim}\,P_1\times\mathrm{colim}\,P_2$$
is an isomorphism; and likewise for any finite number of factors.
\end{lemma}
\begin{proof}
With $C_i=\mathrm{colim} P_i$ one has $$\mathrm{colim}(P_1\otimes P_2)\cong\mathrm{colim}_{b_1}(\mathrm{colim}_{b_2}(P_1b_1\times P_2b_2)) \cong  \mathrm{colim}_{b_1}(P_1b_1\times C_2)\cong C_1\times C_2.  $$
\end{proof}

Trading Lemma \ref{Lemma3} for Lemma \ref{HSLemma3}, with the same calculation as in the proof of Theorem \ref{VolgerTheorem+} one obtains the following strengthening 
  of Howlett's and Schumacher's Theorem \cite{HS72}:

\begin{theorem}\label{HSTheorem} 
Let $\CE$ be as in Lemma {\em \ref{HSLemma3}}, and let $F:\CB\to\CC$ be any functor, with the small category $\CB$ having finite products that are preserved 
 by the functor $X:\CB\to\CE$. Then also $\mathrm{Lan}_FX$ preserves finite products.
\end{theorem}

\begin{corollary}
For a morphism $J:\CB\to\CA$ of algebraic theories $A: \CS_{\aleph_0}^{\op}\to \CA$ and $B:\CS_{\aleph_0}^{\op}\to\CB$ (so that $A=JB$), the left adjoint  of the induced algebraic functor $\mathsf{Alg}(\CA)\to\mathsf{Alg}(\CB)$ assigns to a $\CB$-algebra $Y$ the $\CA$-algebra $\mathrm{Lan}_JY$.
\end{corollary}

\begin{remark}
Theorem \ref{HSTheorem} allows us to produce $\CA$-algebras from $\CB$-algebras even for arbitrary functors $J:\CB\to \CA$, no commutation with $A$ and $B$ or product preservation required. For example, for any finitary algebraic theory $\CA$ we have the functor 
$$J:\CA\longrightarrow\CA,\quad (\omega:A^n\to A^k)\longmapsto( \omega\times 1_{A^1}:A^{n+1}=A^n\times A^1\to A^{k+1}=A^k\times A^1),$$
which fails to preserve products, but still gives for every $\CA$-algebra $X$ the $\CA$-algebra ${X_+:=\mathrm{Lan}_J X}$. In case of the trivial theory $\CA=\CS_{\aleph_0}^{\op}$, a quick computation (for which we thank Francisco Marmolejo) shows that, for any set $X$, one has $X_+\cong X+1$.
\end{remark}

\subsection{Comparison with the Borceux-Day paper}
In the context of categories enriched in a symmentric monoidal-closed category $\CV$, with their Theorem 1.5 Borceux and Day \cite{BorceuxDay77} express the finite-product preservation by left Kan extensions equivalently as the compatibility of coends with finite products. More precisely, they show that, given $\CV$-categories $\CB$ and $\CD$ with finite $\CV$-products, $\CB$ small and $\CD$ with small $\CV$-colimits,  the following conditions are equivalent:
\begin{itemize}
\item[(i)]  For all $\CV$-functors $F:\CB\to\CC$ and $X:\CB\to\CD$, if $X$ preserves finite $\CV$-products, so does ${\mathrm{Lan}}_FX:\CC\to\CD$.
\item[(ii)] For all $\CV$-functors $H_1, H_2:\CB^{\op}\to\CV$ and $X:\CB\to\CD$, if $X$ preserves finite $\CV$-products, then the canonical morphism
$$ \int^{(b_1,b_2)}(H_1b_1\times H_2b_2)\otimes X(b_1\times b_2)\longrightarrow \left(\int^{b_1} H_1b_1\otimes Xb_1\right)\times \left(\int^{b_2}H_2b_2\otimes Xb_2\right)$$
is an isomorphism.
\end{itemize}
In their Example 3.1, Borceux and Day \cite{BorceuxDay77} show that, for $\CV$ {\em cartesian} closed, $\CD:=\CV$ satisfies these equivalent conditions. In fact, their short supporting calculation entails a Fubini-type analogue of the proof of Lemma \ref{HSLemma3}:
\begin{align*}
 \int^{(b_1,b_2)}(H_1b_1\times H_2b_2)\otimes X(b_1\times b_2) &\cong  \int^{b_1}\!\!\!\!\int^{b_2}(H_1b_1\times H_2b_2)\otimes (Xb_1\times Xb_2) \\
 &\cong\int^{b_1}H_1b_1\otimes\int^{b_2}H_2b_2\otimes(Xb_1\times Xb_2) \\
 & \cong\int^{b_1}H_1b_1\otimes\left(Xb_1\times\int^{b_2}H_2b_2\otimes Xb_2\right)\\
&\cong\left(\int^{b_1} H_1b_1\otimes Xb_1\right)\times \left(\int^{b_2}H_2b_2\otimes Xb_2\right)\,.
  \end{align*}
  In particular, Theorem \ref{HSTheorem} follows. 
  
  For $\CV$ just monoidal closed (not necessarily cartesian closed) with finite products and small co-\\limits, Borceux and Day \cite{BorceuxDay80} take the satisfaction of property (ii) as part of the defining conditions for $\CV$ to be a {\em $\pi$-category} and then formulate property (i) with $\CD=\CV$ as their Proposition~{2.2.1}. This is the starting point for their  study of Lawvere-style of $\CV$-enriched algebraic theories, categories and functors, which includes the existence proof of free algebras and, in fact, of left adjoints to algebraic functors in the enriched setting.

$$ $$
\begin{tabular}{lcl}
Mat\'{i}as Menni &\qquad& Walter Tholen\\
Conicet and Centro de Matem\'atica (CMaLP) &\qquad& Department of Mathematics and Statistics\\
Universidad Nacional de La Plata &\qquad& York University \\
La Plata &\qquad& Toronto ON \\
Argentina &\qquad& Canada\\
matias.menni@gmail.com &\qquad& tholen@yorku.ca \\
\end{tabular}



 { \parindent 0pt
     \parskip 2ex
    \def\enotesize{\normalsize}
     \theendnotes   }


\begin{thebibliography}{10}

\bibitem{ARV11}
J.~Ad\'{a}mek, J.~Rosick\'{y}, E.M.~Vitale.
\newblock{\em Algebraic Theories}.
\newblock{Cambridge University Press, Cambridge, 2011.}
  
 \bibitem{Berthiaume69}
 P.~Berthiaume.
 \newblock{The functor evaluation.}
 \newblock{\em Lecture Notes in Mathematics} 106 (pp.13--63), Springer-Verlag, New York, 1969.
 
 \bibitem{Borceux76}
 F.~Borceux.
 \newblock{\em Universal algebra in a closed category}.
 \newblock{Preprint, Universit\'{e} Catholique de Louvain, Louvain-la-Neuve, 1976.}
 
 \bibitem{Borceux94}
 F.~Borceux.
 \newblock{\em Handbook of Categorical Algebra I}.
 \newblock{Cambridge University Press, Cambridge, 1994.}
 
 \bibitem{BorceuxDay77}
 F.~Borceux and B.~Day.
 \newblock{On product preserving Kan extensions}.
 \newblock{\em Bulletin of the Australian Mathematical Society} 12:291--296, 1977.
 
 \bibitem{BorceuxDay80}
 F.~Borceux and B.~Day.
 \newblock{Universal algebra in a closed category}
 \newblock{\em Journal of Pure and Applied Algebra} 16:133--147, 1980.
 
  \bibitem{GU71}
  P.~Gabriel and F.~Ulmer.
  \newblock{ Lokal pr\"{a}sentierbare Kategorien}.
  \newblock {\em Lecture Notes in Mathematics} 221, Springer-Verlag, Berlin, 1971.
  
\bibitem{HS72}
C.~Howlett and D.~Schumacher.
\newblock{Free finitary algebras in a cocomplete cartesian closed category.}
\newblock{\em Canadian Mathematical Bulletin} 15(3):373--374, 1972.


\bibitem{KP2012}
P.~Karazeris and G.~Protsonis.
\newblock{Left {Kan} extensions preserving finite products}.
\newblock{\em Journal of Pure and Applied Algebra} 216(8-9):2014--2028,2012.


\bibitem{KellyLack93}
G.M.~Kelly and S.~Lack.
\newblock{Finite-product-preserving functors, Kan extensions
and strongly-finitary 2-monads}.
\newblock{\em Applied Categorical Structures} 1(1):85--94, 1993.

\bibitem{Lawvere63} 
F.W.~Lawvere.
\newblock{\em Functorial semantics of algebraic theories.}
\newblock{Dissertation, Columbia
University, New York, 1963.} 

\bibitem{Lawvere63a} 
F.W.~Lawvere.
\newblock{Functorial semantics of algebraic theories.}
\newblock{\em Proceedings of the National Academy of Sciences} 50:869--872, 1963.

\bibitem{Lawvere65} 
F.W.~Lawvere.
\newblock{\em Algebraic theories, algebraic categories, and algebraic functors.} 
\newblock{Proceedings of the 1963 International Symposion at Berkeley (pp. 413--418), }
\newblock{North-Holland Publishing Company, Amsterdam, 1965.}

\bibitem{Lawvere04}
F.W.~Lawvere.
\newblock{Functorial semantics of algebraic theories and
    Some algebraic problems in the context of functorial semantics of algebraic theories}.
    \newblock{\em Reprints in Theory and Applications of Categories} 5:1--121, 2004.

\bibitem{Linton66}
F.E.J.~ Linton.
\newblock{\em Some aspects of equational categories.} 
\newblock{Proceedings of the Conference on Categorical Algebra, La Jolla 1965 (pp. 84--94).}
\newblock Springer-Verlag, NewYork, 1966.

\bibitem{ML71}
S.~Mac Lane.
\newblock{\em Categories for the Working Mathematician}.
\newblock{Springer-Verlag, New York, 1971. Second Edition: 1994.}

\bibitem{MLM92}
S.~Mac Lane and I.~Moerdijk.
\newblock{\em Sheaves in Geometry and Logic}.
Springer-Verlag, New York, 1992.

\bibitem{LucyshynParker24}
R.B.B.~Lucyshyn-Wright and J.~Parker.
\newblock{Enriched structure-semantics adjunctions and monad-theory equivalences for subcategories of arities}.
\newblock{\em Theory and Applications of Categories}  41(52):1873-1918, 2024.

\bibitem{Mitchell65}
B.~Mitchell.
\newblock{\em Theory of Categories.}
\newblock{Academic Press, New York and London, 1965.}

\bibitem{Pareigis69}
B.~Pareigis.
\newblock{\em Kategorien und Funktoren}.
\newblock B.G.~Teubner, Stuttgart, 1969.

\bibitem{Pareigis70}
B.~Pareigis.
\newblock{\em Categories and Functors}.
\newblock{Academic Press, Cambridge MA, 1970.}

\bibitem{PedRov04}
M.C.~Pedicchio and F.~Rovatti.
\newblock{\em Algebraic Categories.} 
\newblock In: M.C.~Pedicchio and W. Tholen (editors), {\em Categorical Foundations}.
\newblock Cambridge University Press, Cambridge, 2004.

\bibitem{Schubert70}
H.~Schubert.
\newblock{\em Kategorien II}.
Springer-Verlag, Berlin, 1970.

\bibitem{Schubert72}
H.~Schubert.
\newblock{\em Categories}.
Springer-Verlag, New York, 1972.

\bibitem{Schumacher70}
D.~Schumacher.
\newblock{Zur Existenz freier Algebren einer $r$-dimensionalen Theorie.}
\newblock{\em Manuscripta Mathematica} 3:227--236, 1970.

\bibitem{Ulmer66}
F.~Ulmer.
\newblock{{\em Dichte Unterkategorien in Funktorkategorien} (Dense subcategories in functor categories).}
\newblock Manuscript, Eidgen\"{o}ssische Technische Hochschule Z\"{u}rich, 1966.

\bibitem{Ulmer68}
F.~Ulmer.
\newblock{Properties of dense and relative adjoint functors.}
\newblock{\em Journal of Algebra} 8:77--95, 1968.

\bibitem{Volger68}
H.~Volger.
\newblock{\"{U}ber die Existenz der freien Algebren}.
\newblock{\em Mathematische Zeitschrift} 106:312--320, 1968.

\bibitem{Volger69}
H.~Volger.
\newblock{Korrekturen zur Arbeit ``\"{U}ber die Existenz der freien Algebren''}.
\newblock{\em Mathematische Zeitschrift} 108:388, 1969.

\end{thebibliography}
\end{document}